\documentclass{amsart}

\usepackage{amscd,amssymb,verbatim,graphicx,stmaryrd,amsthm,color,amsmath} 
\usepackage[colorlinks]{hyperref}
\usepackage{cleveref}
\usepackage{tikz-cd}
\usepackage[new]{old-arrows}
\usepackage{tikz}
\usetikzlibrary{arrows}
\usepackage{caption}

\input xy    
\xyoption{curve}  
\xyoption{arc} 
\xyoption{frame} 
\xyoption{tips} 
\xyoption{arrow} 
\xyoption{color} 

\usepackage{mathpazo} 

\newtheorem{theorem}{Theorem}

\newtheorem{innercustomgeneric}{\customgenericname}
\providecommand{\customgenericname}{}
\newcommand{\newcustomtheorem}[2]{%
  \newenvironment{#1}[1]
  {%
   \ifdefined\crefalias\crefalias{innercustomgeneric}{#2}\fi
   \renewcommand\customgenericname{#2}%
   \renewcommand\theinnercustomgeneric{##1}%
   \innercustomgeneric
  }
  {\endinnercustomgeneric}%
  \ifdefined\crefname\crefname{#2}{#2}{#2s}\fi
}

\newcustomtheorem{customthm}{Theorem}
\newcustomtheorem{customques}{Question}
\newtheorem{lemma}[theorem]{Lemma}  
\newtheorem{proposition}[theorem]{Proposition}
\newtheorem{corollary}[theorem]{Corollary}
\newtheorem{remark}[theorem]{Remark}
\newtheorem*{remark*}{Remark}

\newtheorem*{definition*}{Definition}

\usepackage{blindtext}
\usepackage{changepage}
\newtheoremstyle{indented}{3pt}{3pt}{}{}{\bfseries}{.}{.5em}{}
\theoremstyle{indented}

\tikzset{pics/.cd,
handle/.style={code={
\draw  (-2,0) coordinate (-lpoint) 
to [out=100, in=300] (-2.8,2) 
to [out=120, in=250] (-3,4) 
to [out=70,in=180] (0,6) 
to [out=0,in=110] (3,4) 
to [out=290,in=60] (2.8,2) 
to [out=240,in=80] (2,0)  coordinate (-rpoint);
\draw (0,.75) .. controls (1,2) and (1,3.25) .. (0,4.25);
\draw (0.2,1) .. controls (-.75,2.25) and (-.75,3.25) .. (0.25,4);
\draw[blue,->, line width=3pt] (0,-pi) coordinate  
to [out=100, in=270] (-1.5,3.25) 
to [out=90, in=180] (0,5) ;
\draw[blue, line width=3pt]  (0,5) coordinate  
to [out=0,in=90] (1.5,3.25) 
to [out=270,in=80] (0,-pi);
\draw[red,->, line width=3pt] (0,-pi) coordinate  
to [out=60, in=230] (1.5,-1) ;
\draw[red, line width=3pt] (1.5,-1) coordinate  
to [out=50, in=250] (2,0) ;
\draw[red, loosely dotted,line width=3pt] (2,0) coordinate  
to [out=100, in=80] (.25,1) ;
\draw[red, line width=3pt] (.25,1) coordinate  
to [out=260, in=90] (0,-pi) ;
\filldraw[black] (0,-pi) circle (6pt);
}}}

\tikzset{pics/.cd,
handle2/.style={code={
\draw  (-2,0) coordinate (-lpoint2) 
to [out=100, in=300] (-2.8,2) 
to [out=120, in=250] (-3,4) 
to [out=70,in=180] (0,6) 
to [out=0,in=110] (3,4) 
to [out=290,in=60] (2.8,2) 
to [out=240,in=80] (2,0)  coordinate (-rpoint2);
\draw (0,.75) .. controls (1,2) and (1,3.25) .. (0,4.25);
\draw (0.2,1) .. controls (-.75,2.25) and (-.75,3.25) .. (0.25,4);
\filldraw[black] (0,-pi) circle (6pt);
}}}

\newcommand{\bb}[1]{\mathbb{#1}}

\newcommand{\defeq}{\mathrel{\mathpalette{\vcenter{\hbox{$:$}}}=}}
\newcommand{\SU}{\mathrm{SU}}

\newcommand{\U}{\mathrm{U}}

\newcommand{\X}{{\mathfrak{X}}}

\newcommand{\C}{{\mathcal{C}}}

\newcommand{\D}{{\mathcal{D}}}

\newcommand{\F}{{\mathcal{F}}}

\newcommand{\R}{{\mathcal{R}}}

\title{Representation varieties and genus-three Torelli maps}

\author[A.\ Bao]{Allen\ Bao}
\address[Bao]{Department of Mathematics, University of Maryland, College Park, MD 20742}
\email{abao12@terpmail.umd.edu}

\author[A.\ Chakraborty]{Anunoy\ Chakraborty}
\address[Chabraborty]{Mathematical Sciences Department, George Mason University, Fairfax, VA 22030}
\email{achakra5@gmu.edu}

\author[D.\ L.\ Duncan]{David\ L.\ Duncan}
\address[Duncan]{Department of Mathematics \& Statistics, James Madison University, Harrisonburg, VA 22801}
\email{duncandl@jmu.edu}

\author[J.\ Larson]{Jordan\ Larson}
\address[Larson]{Department of Mathematics, Purdue University, West Lafayette, IN 47907}
\email{tlarson.jordan@gmail.com}

\author[K.\ McBride]{Kelson\ McBride}
\address[McBride]{Department of Mathematics and Statistics, Binghamton University, Binghamton, NY 13902}
\email{kelsondmcbride@gmail.com}

\begin{document}

\begin{abstract}
We consider the family of Torelli homeomorphisms on a genus-three surface given by powers of a fixed bounding pair map. For each such homeomorphism $\phi$ we determine the number of connected components of the fixed point set of the induced map on the representation variety of the surface, as well as the number of connected components of the representation variety of the mapping torus of $\phi$. 
\end{abstract}

\maketitle

In the latter half of the 20th century, representation varieties became increasingly commonplace in the low-dimensional topologist's toolkit, appearing in various capacities from deformation spaces for geometric structures \cite{Thurston} to manifold invariants arising in gauge-theoretic constructions \cite{CS} \cite{AB} \cite{AM} \cite{Taubes} \cite{Floer}. Even relatively simple topological features of the $\SU(2)$-representation variety $\R(Y)$ and its descendant, the character variety $\X(Y)$, can detect interesting features of the underlying manifold $Y$. For example, when $Y$ is a 3-manifold, if $\R(Y)$ is connected, then the Chern--Simons invariant of $Y$ is trivial. In the category of 3-manifolds, it is generally understood that $\R(Y)$ encodes considerable information about the topology of $Y$ itself, but determining precisely what information is encoded is still an active area of research \cite{BS} \cite{BLSY} \cite{LP-CZ} \cite{SZ2} \cite{SZ}.

As exemplified through work of Thurston \cite{Thurston2} and others, a particularly rich class of 3-manifolds are those that can be written $Y = \Sigma_\phi$ as the mapping torus of a homeomorphism $\phi: \Sigma \rightarrow \Sigma$ of a surface $\Sigma$ (unless otherwise specified, all manifolds are connected, oriented, and compact without boundary). Recall that the Torelli group $\mathcal{I}(\Sigma)$ consists of (the isotopy classes of) those homeomorphisms of $\Sigma$ that induce the identity on homology. When $\Sigma$ has genus two it was shown by McCullough--Miller \cite{MM} that $\mathcal{I}(\Sigma)$ is not finitely-generated; they also gave a concrete generating set in terms of Dehn twists $T_\gamma$ about certain bounding curves $\gamma \subseteq \Sigma$. It was shown by Daveler \cite{D} that $\R(\Sigma_\phi)$ is connected for each $\phi$ in the McCullough--Miller generating set. On the other hand, it is trivially the case that in genus one all Torelli $\phi$ have $\R(\Sigma_\phi)$ connected, so this makes one wonder: 

\begin{customques}{1}\label{ques:1} 
Is $\R(\Sigma_\phi)$ connected when $\phi$ is Torelli?
\end{customques}

\noindent Here we show that the answer is: \emph{No!}

To elaborate, when the genus of $\Sigma$ is at least three, Dehn twists about bounding curves are not sufficient to generate $\mathcal{I}(\Sigma)$. For example, one cannot \cite{P} get \emph{bounding pair maps}, which are maps of the form $T_{\gamma_1} \circ T_{\gamma_2}^{-1}$, with $\gamma_1, \gamma_2$ disjoint, simple curves in $\Sigma$ with the property that $\gamma_1 \cup \gamma_2$ bounds. With this, our first main result is as follows.

\begin{customthm}{A}\label{thm:A}
Suppose $\Sigma$ has genus three and $\Phi = T_{\gamma_1}\circ T_{\gamma_2}^{-1}$ is the bounding pair map with $\gamma_1, \gamma_2$ as illustrated in Figure \ref{fig:1}. For each $n \in \bb{Z}$, the representation variety $\R(\Sigma_{\Phi^n})$ and character variety $\chi(\Sigma_{\Phi^n})$ each have $n^2 + 1$ (resp. $n^2$) connected components when $n$ is even (resp. odd). 
\end{customthm}

Returning to \cref{ques:1}, it follows from \cref{thm:A} that, e.g., the map $\Phi^2 = \Phi \circ \Phi$ is a Torelli map with $\R(\Sigma_{\Phi^2})$ not connected. Our proof, which appears in Section \ref{sec:ProofOfTheoremA}, is constructive and gives a cover of the representation variety from which much of its topological and geometric structure can be seen. For example, from it one can construct explicit representations in each connected component, and one can also detect the dimension of the smooth stratum. Our analysis also highlights the components (there is only one) containing abelian representations.

\begin{figure}
\resizebox{7cm}{6cm}{
\begin{tikzpicture}
\pic (upper2) at (0,pi) {handle2};
\pic[rotate=120] (bl2) at (30:-pi) {handle2};
\pic[rotate=-120] (br2) at (150:-pi) {handle2};
\draw (upper2-lpoint2) to[out=280,in=20] (bl2-rpoint2);
\draw (upper2-rpoint2) to[out=260,in=160] (br2-lpoint2);
\draw (bl2-lpoint2) to[out=45,in=135] (br2-rpoint2);
\draw[line width=3pt] (.25,7) coordinate  
to [out=50, in=310] (.25,9.15) ;
\draw[loosely dotted,line width=3pt] (.25,7) coordinate  
to [out=140, in=220] (.25,9.15) ; 
\draw[line width=3pt] (.25,4.25) coordinate  
to [out=320, in=40] (.25,-2.65) ;
\draw[loosely dotted,line width=3pt] (.25,4.25) coordinate  
to [out=220, in=140] (.25,-2.65) ; 
\node (x0) at (.8,-.3) [ label={\Huge $x_0$}] {};
\node (gamma1) at (1.2,7.6) [ label={\Huge $\gamma_1$}] {};
\node (gamma2) at (1.85,-2) [ label={\Huge $\gamma_2$}] {};
\end{tikzpicture}}
\caption{Illustrated here are the (unoriented) curves $\gamma_1, \gamma_2$ relative to which the bounding pair map $\Phi = T_{\gamma_1} \circ T_{\gamma_2}^{-1}$ is formed. Note that the basepoint $x_0$ appears here to the \emph{left} of $\gamma_2$ (this is the same basepoint from Figure \ref{fig:2}, and appears on the ``$\alpha_2$-$\beta_2$'' side of $\gamma_2$). Since $x_0$ does not lie on either of $\gamma_1$ or $\gamma_2$, we can define the Dehn twists $T_{\gamma_i}$ to have support near $x_0$, and thus $\Phi(x_0) = x_0$.}\label{fig:1}
\end{figure}
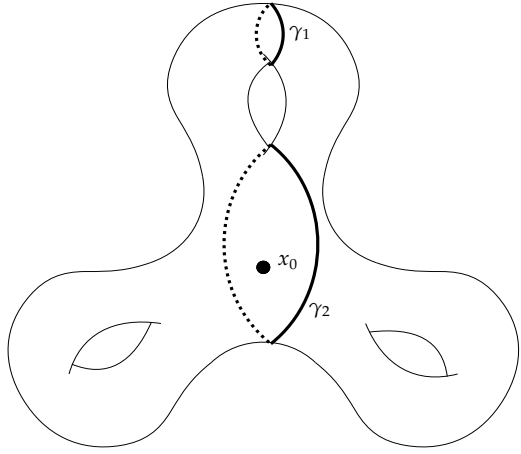

Closely related to the representation variety of the 3-manifold $\Sigma_{\phi}$ is the fixed point set $\mathrm{Fix}_\R(\phi^{*})$ of the map $\phi^{*}: \R(\Sigma) \rightarrow \R(\Sigma)$ given by pullback by $\phi$; there is also its character variety cousin $\mathrm{Fix}_\X(\phi^{*}) \subseteq \X(\Sigma)$. At the character variety level, there is a natural surjection $\X(\Sigma_\phi) \twoheadrightarrow \mathrm{Fix}_\X(\phi^{*})$; see Lemma \ref{lem:charrest}. It is through this surjection that techniques in symplectic geometry \cite{Atiyah} \cite{DS} and ergodic theory \cite{G} can be introduced into low-dimensional topology. In Section \ref{sec:ProofOfTheoremB} we will use this surjection to deduce the size of $\pi_0(\mathrm{Fix}_\X(\phi^{*}))$ from that of $\pi_0(\X(\Sigma_\phi))$:

\begin{customthm}{B}\label{thm:B}
Let $\Phi$ and $n$ be as in \cref{thm:A}. The fixed point set $\mathrm{Fix}_\X((\Phi^n)^*)$ has $n^2/2 + 1$ (resp. $(n^2 + 1)/2$) connected components when $n$ is even (resp. odd). 
\end{customthm}

At the representation variety level, in place of the surjection $\X(\Sigma_\phi) \twoheadrightarrow \mathrm{Fix}_\X(\phi^{*})$ we have an embedding $\mathrm{Fix}_\R(\phi^{*}) \hookrightarrow \R(\Sigma_\phi)$, but this generally does \emph{not} yield a straight-forward way to compute $\pi_0(\mathrm{Fix}_\R(\phi^{*}))$ from $\pi_0(\R(\Sigma_\phi))$. As such, a separate analysis of $\mathrm{Fix}_\R(\phi^{*})$ is often required to determine its connected components. We will return to this in a moment.

The $\SU(2)$-representation variety $\R(\Sigma)$ can be viewed as a nonlinear version of the first cohomology of $\Sigma$ \cite[\S 3.1]{Atiyah}. Given that Torelli homeomorphisms $\phi$ act trivially on $H^1(\Sigma)$, one might expect that the action of $\phi$ on $\R(\Sigma)$ is correspondingly simple. For example, an analogue of \cref{ques:1} might be:

\begin{customques}{2}\label{ques:2} 
Is $\mathrm{Fix}_\R(\phi^*)$ connected when $\phi$ is Torelli?
\end{customques}

Once again, we show the answer is \emph{no} by carrying out the aforementioned analysis of the connected components of $\R(\Sigma)$.

\begin{customthm}{C}\label{thm:C}
Let $\Phi$ and $n$ be as in \cref{thm:A}. The fixed point set $\mathrm{Fix}_\R((\Phi^n)^*)$ has $n^2/2 + 1$ (resp. $(n^2 + 1)/2$) connected components when $n$ is even (resp. odd). 
\end{customthm}

The proof of \cref{thm:C} is given in Section \ref{sec:ProofOfTheoremC}, after we review the relevant preliminary material.

\begin{remark*}
(a) All of the spaces we consider are varieties or manifolds and hence are locally path connected. As such, the connected components are the path components and so there is no need to distinguish between these in the notation. 

\medskip

(b) A key feature of $\SU(2)$ we use is the following: If $A \in \SU(2)$ is not central, then its centralizer is abelian (in fact, it is isomorphic to $S^1$, but we don't use this as often). Much of our analysis extends directly to other Lie groups with this property, such as $\SU(n)$ and $\U(n)$; we leave the details of such extensions to the interested reader.

\medskip

(c) We can see from Theorems \ref{thm:B} and \ref{thm:C} that $\mathrm{Fix}_\R((\Phi^n)^*)$ and $\mathrm{Fix}_\X((\Phi^n)^*)$ have the same number of connected components. We caution the reader that the number of components of $\mathrm{Fix}_\R(\phi^*)$ and $\mathrm{Fix}_\X(\phi^*)$ need not be equal for other functions $\phi$. 
\end{remark*}

\medskip

\textbf{Acknowledgements}: This work was completed as part of an REU hosted at James Madison University (JMU) and supported by the National Science Foundation (NSF) through NSF Grant Number DMS 2349593. We are grateful to both JMU and the NSF for their support.

\section{Preliminary material}

\subsection{Mapping tori}

Suppose $X$ is a topological space and $\phi: X \rightarrow X$ is a homeomorphism. The \emph{mapping torus} of $\phi$ is the quotient space 
$$X_\phi \defeq \frac{[0, 1 ] \times X}{(1, x) \sim (0, \phi(x))}.$$
Fix a basepoint $x_0 \in X$ and identify it with the image of $(0, x_0)$ in $X_\phi$. Assume that $\phi(x_0) = x_0$; when $X$ is a connected manifold this can always be arranged by replacing $\phi$ by a homeomorphism isotopic to $\phi$. The fundamental group $\pi_1(X_\phi, x_0)$ has a presentation 
\begin{equation}\label{eq:presformt}
\Big\langle \tau, \gamma \; \Big| \;   \phi_* \gamma = \tau^{-1} \gamma \tau , \;\;\forall \gamma \in \pi_1(X, x_0)\Big\rangle.
\end{equation}

We will be particularly interested in the case where $X = \Sigma$ is a surface of genus $g$. In this case we will work relative to a \emph{standard presentation} for $\pi_1(\Sigma, x_0)$
$$\Big\langle \alpha_1, \beta_1, \ldots, \alpha_g, \beta_g \; \Big| \; \prod_{i = 1}^g [\alpha_i, \beta_i] = 1 \Big\rangle$$
where $[\alpha, \beta] = \alpha \beta \alpha^{-1} \beta^{-1}$ is the group commutator and $1$ is the identity. In the case $g = 3$, explicit representatives for the generators in the standard presentation are illustrated in Figure \ref{fig:2}. It follows from the above discussion that, when $g \geq 1$, the fundamental group $\pi_1(\Sigma_\phi)$ has presentation
$$\Big\langle \tau, \alpha_1, \beta_1, \ldots, \alpha_g, \beta_g \; \Big| \; \prod_{i =1}^g [\alpha_i , \beta_i] = 1, \;\;  \phi_* \alpha_i = \tau^{-1} \alpha_i \tau, \;\; \phi_* \beta_i = \tau^{-1} \beta_i \tau, \;\; \forall i\Big\rangle.$$

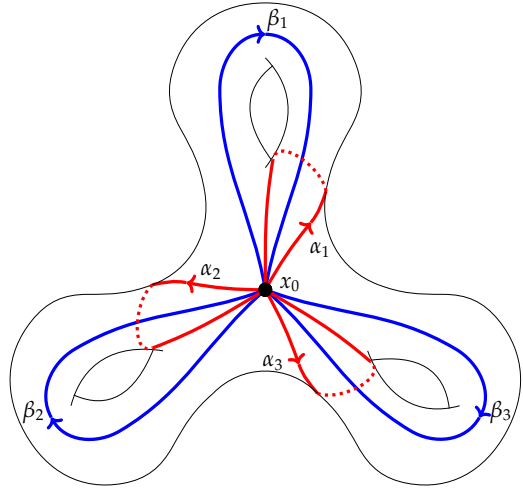
\begin{figure}
\resizebox{7cm}{6.5cm}{
\begin{tikzpicture}
\pic (upper) at (0,pi) {handle};
\pic[rotate=120] (bl) at (30:-pi) {handle};
\pic[rotate=-120] (br) at (150:-pi) {handle};
\draw (upper-lpoint) to[out=280,in=20] (bl-rpoint);
\draw (upper-rpoint) to[out=260,in=160] (br-lpoint);
\draw (bl-lpoint) to[out=45,in=135] (br-rpoint); 
\node (x0) at (.8,-.3) [ label={\Huge $x_0$}] {};
\node (alpha1) at (1.8,.8) [ label={\Huge $\alpha_1$}] {};
\node (beta1) at (.4,8) [ label={\Huge $\beta_1$}] {};
\node (alpha2) at (-1.8,.1) [ label={\Huge $\alpha_2$}] {};
\node (beta2) at (-7.65,-4.5) [ label={\Huge $\beta_2$}] {};
\node (alpha3) at (.25,-2.8) [ label={\Huge $\alpha_3$}] {};
\node (beta3) at (7.75,-4.5) [ label={\Huge $\beta_3$}] {};
\end{tikzpicture}}
\caption{Illustrated here is our generating set $\alpha_i, \beta_i$ for the fundamental group of our genus-three surface $\Sigma$. These all have basepoint $x_0$ (in the middle of the figure) and, with the orientations as indicated, these generators satisfy the relation $\prod_{i = 1}^3 [\alpha_i, \beta_i ] = 1$. As such, these give a standard presentation for $\pi_1(\Sigma, x_0)$.}
\label{fig:2}
\end{figure}

\subsection{The representation and character varieties}

Let $G$ be a compact Lie group and $\pi$ a finitely-presented group. The \emph{$G$-representation variety of $\pi$} is the set
$$\R(\pi, G) \defeq \hom(\pi, G)$$
of group homomorphisms from $\pi$ to $G$. We endow this with the compact-open topology, relative to the discrete topology on $\pi$. Given a presentation of the form
$$\pi = \langle \gamma_1, \ldots, \gamma_J \; \vert \; R_k(\gamma_1, \ldots, \gamma_K) = 1,\;\; 1 \leq k \leq K \rangle,$$
we have an embedding
\begin{equation}\label{eq:ident}
\R(\pi, G) \longrightarrow G^J, \hspace{1cm} \rho \longmapsto (\rho(\gamma_1), \ldots, \rho(\gamma_J)).
\end{equation}
This has image the set of $(g_1, \ldots, g_J) \in G^J$ satisfying $R_k(g_1, \ldots, g_J) = 1$ for $1 \leq k \leq K$ (the same relations that appear in our presentation of $\pi$). We will identify $\R(\pi, G)$ with its image under (\ref{eq:ident}), often omitting direct reference to the map (\ref{eq:ident}); we trust the chance for confusion is minimal since our presentation will be fixed at the outset. 

The group $G$ acts on $\R(\pi, G)$ by conjugation on the codomain, and we denote the quotient space by
$$\X(\pi, G) \defeq \R(\pi, G) / G,$$
which we call the \emph{character variety}. The map (\ref{eq:ident}) is equivariant relative to the action of $G$ on $G^J$ by diagonal conjugation, and so (\ref{eq:ident}) descends to give a homeomorphism of $\X(\pi, G)$ onto the set of $G$-equivalence classes of tuples $(g_1, \ldots, g_J) \in G^J$ satisfying the same relations $R_k$ as above. In particular, when $G$ is connected the projection $\R(\pi, G) \rightarrow \X(\pi, G)$ induces a bijection
$$\pi_0(\R(\pi, G)) \cong \pi_0(\X(\pi, G)).$$

We will be particularly interested in the case where $\pi = \pi_1(X, x_0)$ is the fundamental group of a compact, connected topological space $X$, in which case we set
$$\R(X, G) \defeq \R(\pi_1(X, x_0), G), \hspace{1cm} \X(X, G) \defeq \X(\pi_1(X, x_0), G).$$
When $G = \SU(2)$ we drop it from the notation: $\R(X) \defeq \R(X, \SU(2))$ and $\X(X) \defeq \X(X, \SU(2))$. 

\subsection{Fixed point sets}\label{sec:FixedPointSets}

Suppose $X$ is a topological space and $\phi: X \rightarrow X$ is a homeomorphism preserving a fixed basepoint $x_0 \in X$. This induces a homeomorphism
$$\phi^*: \R(X, G) \longrightarrow \R(X, G) \hspace{1cm} \rho \longmapsto \rho \circ \phi_*$$
called the \emph{pullback}, where $\phi_*$ is the isomorphism of $\pi_1(X, x_0)$ associated to $\phi$. The pullback is equivariant relative to the action of $G$ and so descends to a map $\X(X, G) \rightarrow \X(X, G)$ on the character variety that we also call the \emph{pullback} and denote by the same symbol $\phi^*$. We will be interested in the fixed point sets of both variations of pullback, which we denote by
$$\begin{array}{rcl}
\mathrm{Fix}_\R(\phi) & \defeq & \left\{ \rho \in \R(X, G) \; \vert \; \phi^* \rho = \rho \right\}\\
\mathrm{Fix}_\X(\phi) & \defeq & \left\{ [\rho] \in \X(X, G) \; \vert \; \phi^* [\rho] = [\rho] \right\}
\end{array}$$
and equip with the subspace topology. Pullback by the fiber inclusion $X \hookrightarrow X_\phi$ gives a map
\begin{equation}\label{eq:Xsurj}
\X(X_\phi, G) \longrightarrow \mathrm{Fix}_\X(\phi).
\end{equation}

\begin{lemma}\label{lem:charrest}
The map (\ref{eq:Xsurj}) is surjective and the fiber over $[\rho] \in \mathrm{Fix}_\X(\phi)$ is homeomorphic to the centralizer of $\rho$ (i.e., the set of elements of $G$ that commute with the image of $\rho$ in $G$).
\end{lemma}

\begin{proof}
If $[\rho] \in \mathrm{Fix}_\X(\phi)$, then there is some $T \in G$ so that $\phi^* \rho = T^{-1} \rho T$. This implies $[T, \rho] \in \X(X_\phi, G)$ is in the fiber of (\ref{eq:Xsurj}) over $[\rho]$ and so (\ref{eq:Xsurj}) is surjective. If $S \in G$ commutes with every element in the image of $\rho$, and $[T, \rho] \in \X(X_\phi, G)$, then $[ST, \rho] \in \X(X_\phi, G)$. It follows that the centralizer of $\rho$ acts on the fibers of (\ref{eq:Xsurj}) over $[\rho]$. This action is free and transitive on this fiber, so the lemma follows. 
\end{proof}

As an example, when $G = \SU(2)$ the fiber of (\ref{eq:Xsurj}) is connected whenever $\rho$ is \emph{abelian}, meaning $\rho$ takes values in an abelian subgroup of $\SU(2)$; the fiber consists of two points when $\rho$ is not abelian. 

As for the representation varieties, we still have an inclusion
$$\mathrm{Fix}_\R(\phi) \longhookrightarrow \R(X_\phi, G) \cong \left\{(T, \rho) \in G \times \R(X, G) \; \Big| \; \phi^* \rho = T^{-1} \rho T\right\}$$
given by sending $\rho$ to $(1, \rho)$, though this will generally \emph{not} be surjective. We find it is convenient to introduce the \emph{extended $\R$-fixed point set}
$$\widetilde{\mathrm{Fix}}_\R(\phi) \defeq \left\{ \rho \in  \R(X, G) \; \Big| \; \exists T \in G, \; \phi^* \rho = T^{-1} \rho T \right\}.$$
The projection $(T, \rho) \mapsto \rho$ induces a surjective map 
\begin{equation}\label{eq:Rsurj}\R(X_\phi, G) \longrightarrow \widetilde{\mathrm{Fix}}_\R(\phi)
\end{equation}
and essentially the same proof given for Lemma \ref{lem:charrest} shows the following.

\begin{lemma}\label{lem:charrest2}
The fiber of (\ref{eq:Rsurj}) over $\rho \in \widetilde{\mathrm{Fix}}_\R(\phi)$ is homeomorphic to the centralizer of $\rho$.
\end{lemma}

The usefulness of the extended $\R$-fixed point set is that it interpolates between the representation variety and the character variety pictures as follows: 
The space $\widetilde{\mathrm{Fix}}_\R(\phi) \subseteq \R(X, G)$ is $G$-invariant and the quotient 
$$\widetilde{\mathrm{Fix}}_\R(\phi) / G \cong \mathrm{Fix}_\X(\phi)$$
can be naturally identified with the fixed point set of $\phi$ on the character variety. In summary, we have a commutative diagram

\begin{tikzcd}
\mathrm{Fix}_\R(\phi) \arrow[r, hookrightarrow] & \R(X_\phi, G) \arrow[r, twoheadrightarrow] \arrow[d, twoheadrightarrow] & \widetilde{\mathrm{Fix}}_\R(\phi) \arrow[d, twoheadrightarrow] \arrow[r, hookrightarrow]& \R(X, G) \arrow[d, twoheadrightarrow]\\
& \X(X_\phi, G) \arrow[r, twoheadrightarrow] &  \mathrm{Fix}_\X(\phi) \arrow[r, hookrightarrow]& \X(X, G)
\end{tikzcd}

\noindent and a useful result.

\begin{proposition}\label{prop:bij}
When $G$ is connected there are bijections
$$\pi_0(\R(X_\phi, G)) \cong \pi_0(\X(X_\phi, G)), \hspace{1cm} \pi_0(\mathrm{Fix}_\X(\phi)) \cong \pi_0(\widetilde{\mathrm{Fix}}_\R(\phi)).$$
\end{proposition}

\subsection{The commutator}

Much of our analysis will involve the fibers of the commutator map 
$$\mu: G \times G \longrightarrow G, \hspace{1cm} (A, B) \longmapsto [A, B].$$
It is a well-known that when $G$ is connected the fibers of $\mu$ are connected. We will address this explicitly when $G = \SU(2)$. 

The fiber $\mu^{-1}(1)$ is the set of commuting pairs in $\SU(2)$. If $A, B \in \SU(2)$ commute, then they belong to the same maximal torus. Maximal tori are connected, by definition, so $A$ and $B$ can be connected by paths through this maximal torus to the identity. It follows that $\mu^{-1}(1)$ is connected. The same argument proves the following. 

\begin{proposition}\label{prop:connectedcommutators}
For each positive integer $n$, the space
$$\left\{(A_1, \ldots, A_n) \in \SU(2)^n \; \vert \; [A_i, A_j] = 1, \; \forall i, j \right\}$$
of mutually commuting tuples is connected. 
\end{proposition}

Away from the identity, we will use the following. 

\begin{proposition}\label{prop:commutator}
The restriction
$$\mu \vert : \mu^{-1}  (\SU(2)\backslash \left\{1 \right\} ) \longrightarrow \SU(2)\backslash \left\{1 \right\} $$
is a smooth fiber bundle with generic fiber diffeomorphic to $S^3$. In particular, all fibers of $\mu$ are connected.
\end{proposition}

\begin{proof}[Proof Sketch]
That $\mu$ is a fiber bundle away from the commuting pairs is an application of the implicit function theorem, together with the observation that all non-identity elements of $\SU(2)$ are regular values for $\mu$; the details are left to the reader. 

Since $\SU(2) \backslash \left\{1 \right\}$ is connected, it follows from the implicit function theorem of the previous paragraph that all fibers of $\mu$ in $\SU(2) \backslash \left\{1 \right\}$ are diffeomorphic. Our task therefore is to determine the manifold type of one such fiber. We will focus on $\mu^{-1}(-1)$. This fiber consists of the pairs $(A, B) \in \SU(2)^2$ with $ABA^{-1} B^{-1}  = -1$. A little juggling shows $A B A^{-1} = -B$ and so, taking traces, we find $\mathrm{tr}(B) = 0$. Likewise $\mathrm{tr}(A) = 0$ and $\mathrm{tr}(AB) = 0$. In fact, the converse holds: $[A,B] = -1$ if and only if $ \mathrm{tr}(A) = \mathrm{tr}(B)= \mathrm{tr}(AB) = 0$. Identifying $\SU(2) \cong S^3$ with the 3-sphere in the usual way (i.e., as the unit length quaternions), the set of trace-zero elements of $\SU(2)$ is identified with a 2-sphere $S^2 \subseteq S^3$. That is, $\mathrm{tr}(A) = \mathrm{tr}(B) = 0$ if and only if $A, B \in S^2$. Likewise, when $\mathrm{tr}(B) = 0$ we have $B^* = - B$ and so $\mathrm{tr}(AB) = 0$ is equivalent to $\mathrm{tr}(A B^*) = 0$; geometrically this means that $A$ and $B$ are orthogonal. Thus $\mu^{-1}(-1)$ consists of the set of orthogonal pairs $(A, B)$ of elements of $S^2$. This is the 3-sphere; for example, projection to the $A$-component is a principal $S^1$-bundle $\mu^{-1}(-1) \rightarrow S^2$ over $S^2$, which can be seen to be the Hopf fibration. 
\end{proof}

\begin{corollary}\label{cor:com}
Every fiber of the commutator map $\mu: \SU(2)^2 \rightarrow \SU(2)$ is connected and nonempty. 
\end{corollary}

\section{Proofs of the main theorems}

Let $\Sigma$ be a closed, connected, oriented surface of genus 3 and write $\alpha_j, \beta_j$ for the loops based at $x_0$ as indicated in Figure \ref{fig:2}; we will not distinguish in the notation between loops based at $x_0$ and their homotopy classes in $\pi_1(\Sigma, x_0)$. Now suppose $\gamma$ is a simple closed curve in $\Sigma$ that \emph{does not contain $x_0$}. We will write $T_\gamma$ for the (left-handed) Dehn twist about $\gamma$, supported in a small neighborhood of $\gamma$ not containing $x_0$ (so a curve approaching $\gamma$ at a point takes a \emph{left} turn under $T_\gamma$). Let $\gamma_1, \gamma_2$ be as in Figure \ref{fig:1} and set $\Phi \defeq T_{\gamma_1} \circ T_{\gamma_2}^{-1}$, which is a homeomorphism of $\Sigma$ fixing $x_0$. 

Fix $n \in \bb{Z}$. In this section we will determine the number of connected components of the $\SU(2)$-representation variety $\R(\Sigma_{\Phi^n})$, character variety $\X(\Sigma_{\Phi^n})$, and fixed point sets $\mathrm{Fix}_\R((\Phi^n)^*)$ and $\mathrm{Fix}_\X((\Phi^n)^*)$. Our analysis relies on the following lemma, which tells us how $\Phi^n$ acts on $\pi_1(\Sigma, x_0)$.

\begin{lemma}\label{lem:phiaction}
If $n \geq 0$, the isomorphism $\Phi_*^n: \pi_1(\Sigma, x_0) \rightarrow \pi_1(\Sigma, x_0)$ acts on the generating set $\alpha_i, \beta_i$ from Figure \ref{fig:2} as follows:
$$\begin{array}{rclcrcl}
\alpha_{1} & \mapsto &  \chi^n \alpha_{1} \chi^{-n} & \hspace{1cm} &  \beta_{1} & \mapsto & \beta_{1} \alpha_{1}^{n} \chi^{-n}\\
         \alpha_{2} & \mapsto & \alpha_{2} &  &\beta_{2} & \mapsto & \beta_{2} \\
        \alpha_{3} & \mapsto & \chi^n \alpha_{3}  \chi^{-n} & &\beta_{3} & \mapsto & \chi^n  \beta_{3}  \chi^{-n}.
        \end{array}$$
        where $\chi \defeq [\alpha_3, \beta_3] \alpha_1$. 
\end{lemma}

Note the asymmetry in the action on $\alpha_1$ versus the action on $\beta_1$; this is a consequence of the asymmetry of our choice of basepoint $x_0$ in Figure \ref{fig:1}. 

\begin{proof}[Proof of Lemma \ref{lem:phiaction}]
Since $\gamma_1$ does not intersect $\alpha_1, \alpha_2, \alpha_3, \beta_2,$ or $\beta_3$, it follows that $T_{\gamma_1}$ acts as the identity on each of these:
$$(T_{\gamma_1})_* \alpha_i = \alpha_i \; \; 1 \leq i \leq 3, \hspace{1cm} (T_{\gamma_1})_* \beta_j = \beta_j \; \; 2 \leq j \leq 3.$$
On the other hand, $\beta_1$ intersects $\gamma_1$ exactly once, and a close analysis of this intersection reveals\footnote{Our concatenation convention for elements of $\pi_1$ is the ``left-to-right convention'': $\alpha \beta$ means that $\alpha$ is traced first, then $\beta$.} $(T_{\gamma_1})_* \beta_1 = \beta_1 \alpha_1$. Likewise, we find that $T_{\gamma_2}$ acts on our generating set as follows:
$$ \begin{array}{rclcrcl}
\alpha_{1} & \mapsto & \chi^{-1} \alpha_{1} \chi &\hspace{.5cm} &  \beta_{1} & \mapsto & \beta_{1} \chi\\
\alpha_{2} & \mapsto & \alpha_{2} && \beta_{2} & \mapsto & \beta_{2}\\ 
\alpha_{3} & \mapsto & \chi^{-1} \alpha_{3}\chi  & & \beta_{3} & \mapsto &\chi^{-1}\beta_{3} \chi.
\end{array}$$
Thus its inverse $T_{\gamma_2}^{-1}$ acts as
$$\begin{array}{rclcrcl}
 \alpha_{1} & \mapsto & \chi \alpha_{1} \chi^{-1} &\hspace{.5cm} & \beta_{1} & \mapsto & \beta_{1} \chi^{-1}\\
   \alpha_{2} & \mapsto & \alpha_{2} && \beta_{2} &\mapsto & \beta_{2}\\
    \alpha_{3} & \mapsto & \chi \alpha_{3}\chi^{-1} && \beta_{3} & \mapsto & \chi \beta_{3}\chi^{-1} .
\end{array}$$
The lemma for $n = 1$ follows from these formulas by composing $ \Phi = T_{\gamma_1} \circ T_{\gamma_2}^{-1}$, while the case for $n = 0$ is trivial. The case for $n \geq 2$ follows from induction (which can be facilitated by the observation that $\chi  = [\alpha_3, \beta_3] \alpha_1$ is fixed under $\Phi$). 
\end{proof}

\subsection{Proof of \cref{thm:C}}\label{sec:ProofOfTheoremC}

Write $\rho = (A_i, B_i)_{i = 1}^3  \in \R(\Sigma) \subseteq \SU(2)^6$ for a generic element of the representation variety, so $\prod_i [ A_i, B_i] = 1$. Our aim is to show that the set 
$$\mathrm{Fix}_\R((\Phi^n)^*) = \left\{ \rho \in \R(\Sigma)\; \vert \; (\Phi^n)^* \rho = \rho \right\}$$
has $\lfloor n^2/ 2 \rfloor + 1$ connected components. Due to the identity
$$(\Phi^n)^{*}\rho =  \rho  \; \Longleftrightarrow \; (\Phi^{-n})^{*} \rho =  \rho $$
we may assume $n \geq 0$. To simplify the notation, we set $\phi \defeq \Phi^n$.

By Lemma \ref{lem:phiaction}, a representation $\rho = (A_i, B_i)_i$ is in $ \mathrm{Fix}_\R(\phi^*)$ if and only if the following hold
$$\prod_i [A_i, B_i] = 1, \hspace{1cm} \begin{array}{rclcrcl}
A_{1}  & = &  X^n A_1 X^{-n}, &  &B_{1}  & = & B_{1} A_{1}^{n} X^{-n},\\
          A_{2}  & = & A_{2}, & & B_{2}  & = & B_{2}, \\
         A_{3}  & = & X^n A_{3}X^{-n}, & & B_{3}  & = & X^n B_{3}  X^{-n},
        \end{array}$$
        where we have set $X \defeq [A_{3}, B_{3}]A_{1}$. Observe that $B_1, A_2,$ and $ B_2$ are only constrained by the product of commutators relation. To highlight this observation, consider the set
$$\D \defeq \left\{(A_1, A_3, B_3) \in \SU(2)^3 \; \Big| \; A_1^n = ([A_3, B_3]A_1)^n, \; [A_3, A_1^n] = [B_3, A_1^n] = 1 \right\}$$
and the projection
$$p_\D: \mathrm{Fix}_\R(\phi^{*}) \longrightarrow \D, \hspace{1cm} (A_i, B_i)_{i = 1}^3 \longmapsto (A_1, A_3, B_3),$$
which is well-defined by the defining equations for the fixed point set. 

\begin{lemma}\label{lem:D}
The projection induces a bijection
$$\pi_0(\mathrm{Fix}_\R(\phi^{*})) \cong \pi_0(\D).$$
\end{lemma}

\begin{proof}
To prove the lemma, it suffices to show that each fiber of the map $p_\D$ is nonempty and connected. The fiber over $(A_1, A_3, B_3)$ of $p_\D$ is homeomorphic to the set 
$$\C_{(A_1, A_3, B_3)} \defeq \left\{(B_1, A_2, B_2) \in \SU(2)^3 \; \Big| \; [A_2, B_2] = [A_1, B_1]^{-1} [A_3, B_3]^{-1}\right\}.$$
Consider the projection 
$$\C_{(A_1, A_3, B_3)} \longrightarrow \SU(2), \hspace{1cm} (B_1, A_2, B_2) \longrightarrow B_1.$$
The fiber over $B_1$ is homeomorphic to the preimage $\mu^{-1}([A_1, B_1]^{-1} [A_3, B_3]^{-1})$, where $\mu$ is the commutator map. By Corollary \ref{cor:com}, each of these fibers is nonempty and connected. It follows that $\C_{(A_1, A_3, B_3)} \cong p_\D^{-1}(A_1, A_3, B_3)$ is nonempty and connected. 
\end{proof}

To prove Theorem \ref{thm:C}, we will show that $\D$ has $\lfloor n^2/2 \rfloor + 1$ connected components. For $\pm \in \left\{+, -\right\}$, define
$$\begin{array}{rcl}
\D_\pm & \defeq & \left\{ (A_1, A_3, B_3) \in \D \; \vert \; A_1^n = \pm 1 \right\}\\
\D_{0} & \defeq & \left\{ (A_1, A_3, B_3) \in \D \; \vert \; A_1^n \neq 1, \; A_1^n \neq -1 \right\}
\end{array}$$
so $\D = \D_+ \cup \D_{0} \cup \D_-$, and this is a disjoint union. Note that $\D_\pm$ is closed, but $\D_{0}$ is not: its closure intersects $\D_+ \cup \D_-$. 

\begin{lemma}\label{lem:componentsR}
\begin{enumerate}
\item[(a)] The closure $\overline{\D}_{0}$ of $\D_{0}$ in $\D$ is connected.
\item[(b)] The subspace $\D_+$ has $(\lfloor n/2  \rfloor + 1)^2$ connected components and exactly $\lfloor n/2  \rfloor + 1$ of these intersect the closure $\overline{\D}_{0}$.
\item[(c)] The subspace $\D_-$ has $(\lfloor (n-1)/2  \rfloor + 1)^2$ connected components and exactly $\lfloor (n-1)/2  \rfloor + 1$ of these intersect the closure $\overline{\D}_{0}$.
\end{enumerate}
\end{lemma}

We will prove the lemma in a moment, but first we show how it finishes the proof of Theorem \ref{thm:C}: The components of $\D_\pm$ not intersecting $\overline{\D}_{0}$ determine distinct components of $\D$---there are $(\lfloor n/2  \rfloor + 1)\lfloor n/2  \rfloor$ such components in the $\pm = +$ case and $(\lfloor (n-1)/2  \rfloor + 1)\lfloor (n-1)/2  \rfloor$ in the $\pm = -$ case. Since $\overline{\D}_{0}$ produces exactly one component in $\D$, it follows that the number of components in $\D$ is
$$(\lfloor n/2  \rfloor + 1)\lfloor n/2 \rfloor + (\lfloor (n-1)/2  \rfloor + 1)\lfloor (n-1)/2  \rfloor + 1 = \lfloor n^2/2 \rfloor + 1,$$
as desired. 

\begin{proof}[Proof of Lemma \ref{lem:componentsR}]
For (a), note that if $(A_1, A_3, B_3) \in \D_{0}$, then $A_3, B_3$ both lie in the centralizer of the \emph{non-central} element $A_1^n$. Such a centralizer in $\SU(2)$ is abelian (its the unique maximal torus containing $A_1^n$), so $A_3$ and $B_3$ commute. It follows that the closure $\overline{\D}_{0}$ consists of those tuples $(A_1, A_3, B_3)$ with the property that $A_1^n, A_3, B_3$ all commute. It is not hard to see that this is connected: Fixing $A_1$ at first, find a path $(A_1, A_3(t), B_3(t))$ from $(A_1, A_3, B_3)$ to $(A_1, 1, 1)$ with $A_3(t), B_3(t)$ in the centralizer of $A_1^n$; this guarantees that $(A_1, A_3(t), B_3(t))$ lies in $\overline{\D}_{0}$ for all $t$. Then if $A_1(t)$ is any path from $A_1$ to $1$, the tuple $(A_1(t), 1, 1)$ is a path in $\overline{\D}_{0}$. Concatenating gives a path in $\overline{\D}_{0}$ from $(A_1, A_3, B_3)$ to $(1, 1, 1)$. 

Next we prove (b); the proof of (c) is similar. Note that if $(A_1, A_3, B_3) \in \D_+$, then 
$$\mathrm{tr}(A_1) = 2\cos(2 \pi k /n), \hspace{1cm} \mathrm{tr}([A_3, B_3] A_1) = 2\cos(2 \pi \ell / n)$$ 
for unique integers $k, \ell$ with $0 \leq k, \ell \leq n/2$. It follows that the assignment
$$(A_1, A_3, B_3) \longmapsto \frac{n}{2 \pi} \Big(\cos^{-1}\big(\mathrm{tr}(A_1)/2\big), \;\cos^{-1}\big(\mathrm{tr}([A_3, B_3] A_1)/2\big)\Big)$$
determines a well-defined continuous map of the form
$$p_+: \D_+ \longrightarrow \left\{(k , \ell)  \in \bb{Z}^2 \; \Big| \; 0 \leq k, \ell \leq \lfloor n/ 2 \rfloor \right\}.$$
We will show that this map induces a bijection on $\pi_0$; this gives the first claim in (b). To show this, it suffices to prove that each fiber of $p_+$ is nonempty and connected. That each fiber is nonempty can be seen as follows: Fix $k, \ell$ in the codomain and let $A_1$ and $X$ be elements of $\SU(2)$ with
$$\mathrm{tr}(A_1) = 2\cos(2 \pi k /n), \hspace{1cm} \mathrm{tr}(X) = 2\cos(2 \pi \ell /n).$$
By Corollary \ref{cor:com}, there are $A_3, B_3 \in \SU(2)$ with $[A_3, B_3] = X A_1^{-1}$, so 
$$p_+(A_1, A_3, B_3) = (k, \ell).$$ 
As for connectivity, suppose $p_+(A_1, A_3, B_3) = p_+(A_1', A_3', B_3')$. Then $\mathrm{tr}(A_1) = \mathrm{tr}(A_1')$ and so $A_1, A_1'$ are conjugate; this implies there is a path $A_1(t)$ from $A_1$ to $A_1'$ with $\mathrm{tr}(A_1(t))$ constant in $t$. Assume for now that $[A_3, B_3] \neq 1$ and $[A_3', B_3'] \neq 1$. Then we can find a path $Y(t)$ from $[A_3, B_3]$ to $[A_3', B_3']$ with $Y(t) \neq 1$ for all $t$ and $\mathrm{tr}(Y(t)A_1(t))$ constant in $t$. By Proposition \ref{prop:commutator}, we can find paths $A_3(t)$ and $B_3(t)$ from $A_3$ and $B_3$ to $A_3'$ and $B_3'$, respectively, with $[A_3(t), B_3(t)] = Y(t)$. Then $(A_1(t), A_3(t), B_3(t))$ is a path in $\D_+$ from $(A_1, A_3, B_3)$ to $(A_1', A_3', B_3')$. 

Next, consider the case where $[A_3, B_3] = [A_3', B_3'] = 1$. By Proposition \ref{prop:connectedcommutators}, the space of commuting pairs is connected, so there is a path $(A_3(t), B_3(t))$ from $(A_3, B_3)$ to $(A_3', B_3')$ with $[A_3(t), B_3(t)] = 1$ for all $t$. Then $(A_1(t), A_3(t), B_3(t))$ is the desired path in this case. 

Now consider the case where exactly one of $[A_3, B_3]$ and $[A_3', B_3']$ equals 1; by relabeling, we may assume $[A_3, B_3] \neq 1$ and $[A_3', B_3'] = 1$. Then we can find a path $Y(t)$, parametrized by $t \in [0, 1]$, with $Y(0) = [A_3, B_3]$, $Y(1) = 1$, and $Y(t) \neq 1$ for $t \in [0, 1)$. Similar to the first case, for $t \in [0, 1)$ we can find $A_3(t)$ and $B_3(t)$ with $[A_3(t), B_3(t)] = Y(t)$. The path $(A_3(t), B_3(t))$ can be chosen to be continuous in $t$ and to converge to a limiting value $(A_3(1), B_3(1))$ as $t \rightarrow 1^-$. This reduces the present case to the previous one where $[A_3, B_3] = [A_3', B_3'] =1$, and thus finishes the proof of the first claim in (b). 

To prove the second claim in (b), recall that the closure of $D_{0}$ consists of those tuples $(A_1, A_3, B_3)$ with $A_1^n, A_3, B_3$ all commuting. The intersection $\D_+ \cap \overline{D}_{0}$ thus consists of those with $A_1^n = 1$ and $[A_3 , B_3] = 1$; this corresponds to those points in $\D_+$ that map to the diagonal under $p_+$, of which there are $\lfloor n/2 \rfloor + 1$ such components. 
\end{proof}

\begin{remark}\label{rem:AbelianRep}
Suppose $\rho = (A_i, B_i)_i \in \R(\Sigma)$ is abelian, so all $A_i, B_i$ commute. This is automatically a fixed point of $\phi^*$ and it has $X = A_1$. This latter identity, via the analysis above involving $p_+$ (and its $\pm = -$ counterpart), implies that $\rho$ belongs to the connected component of $\mathrm{Fix}_\R(\phi^*)$ containing $\D_{0}$ (i.e., the connected component of the trivial representation). From this we deduce that if $\rho \in \mathrm{Fix}_\R(\phi^*)$ is not in the component containing the trivial representation, then $\rho$ is not an abelian representation. 
\end{remark}


\subsection{Proof of \cref{thm:A}}\label{sec:ProofOfTheoremA}

As above, we set $\phi \defeq \Phi^n$. Our task is to prove that the representation variety $\R(\Sigma_{\phi})$ of the 3-manifold $\Sigma_{\phi}$ has $n^2 + 1$ (resp. $n^2$) connected components when $n$ is even (resp. odd). By Lemma \ref{lem:phiaction}, we can identify $\R(\Sigma_{\phi})$ with the set of $(T, (A_i, B_i)_{i = 1}^3 ) \in \SU(2)^7$ satisfying 
\begin{subequations}
\label{eq:fixed}
 \begin{align}
 \prod_{i = 1}^3 \left[A_i, B_i\right] = & \; 1 \label{eq:fixed0}\\ 
 \left[A_1, X^n \right] = & \;1   \label{eq:fixed1}\\
 \left[B_1^{-1} , T^{-1}\right]  = & \;A_1^n X^{-n}    \label{eq:fixed2}\\
\left[A_2, T \right]  =  \left[B_2, T \right]  = &  \;1 \label{eq:fixed3}\\
\left[A_3, TX^n \right]  =  \left[B_3, TX^n \right]  =  & \;1 \label{eq:fixed4}
         \end{align}
         \end{subequations}
where we have set $X \defeq [A_3, B_3] A_1 $. Restricting to those tuples with $T = \pm 1$, we recover a subset homeomorphic to the fixed point set $\mathrm{Fix}_\R(\phi^{*})$ considered in the previous section. We thus need to understand the subset where $T$ is not central in $\SU(2)$. We split this up into two pieces: For $\pm \in \left\{+, - \right\}$ write $\F_{\pm}$ for the set of $(T, (A_i, B_i)) \in \R(\Sigma_{\phi})$ with $TX^n = \pm 1$ but $T \notin Z(\SU(2))$, and also write $\F_{0}$ for the set of $(T, (A_i, B_i)) \in \R(\Sigma_{\phi})$ with $TX^n, T \notin Z(\SU(2))$. We thus can write our representation variety as a disjoint union
$$\R(\Sigma_{\phi}) = \Big(\left\{1 \right\} \times \mathrm{Fix}_\R(\phi^{*})\Big) \cup \Big(\left\{-1 \right\} \times \mathrm{Fix}_\R(\phi^{*})\Big)  \cup \F_+ \cup \F_0 \cup \F_{-}.$$
We note that $\left\{ \pm 1\right\} \times \mathrm{Fix}_\R(\phi^{*})$ is closed in $\R(\Sigma_{\phi})$, but $\F_\pm$ and $\F_{0}$ are not. We begin by analyzing the closure $\overline{\F}_\pm$ of $\F_\pm$, which non-trivially intersects $\left\{ 1, -1 \right\} \times \mathrm{Fix}_\R(\phi^{*})$ (this closure does not intersect $\F_{0}$). 

\begin{lemma}\label{lem:pm}
The closure $\overline{\F}_{\pm}$ is connected and contains the representations $ (1, (1, 1)_{i = 1}^3)$ and $ (-1, (1, 1)_{i = 1}^3)$. For $\epsilon \in \left\{1, -1 \right\}$, the intersection $\overline{\F}_\pm \cap \left\{\epsilon \right\} \times \mathrm{Fix}_\R(\phi^{*})$ lies in the component of $\left\{ \epsilon \right\} \times \mathrm{Fix}_\R(\phi^{*})$ containing $(\epsilon, (1, 1)_{i = 1}^3)$. 
\end{lemma}

\begin{proof}
Since $TX^n = \pm 1$ is central for elements of ${\F}_{\pm}$, it follows from (\ref{eq:fixed}) that ${\F}_{\pm}$ consists of those tuples with
\begin{subequations}
\label{eq:Xtriv}
 \begin{align}
 \prod_i [A_i, B_i] = &\; 1\label{eq:Xtriv0}\\
 B_1^{-1} T^{-1} B_1 = &\; \pm A_1^n \label{eq:Xtriv1}\\
[A_2, T] = [B_2, T]   = &\; 1 \label{eq:Xtriv2}.
\end{align}
         \end{subequations}
         together with open condition that $T$ is not central in $\SU(2)$. Since $T$ is not central, it follows from (\ref{eq:Xtriv2}) that $A_2, B_2$ commute, and so (\ref{eq:Xtriv0}) reduces to
$$[A_3, B_3] = [B_1, A_1].$$
The identity $TX^n = \pm$ can then be written as
$$\pm 1 = TX^n = T([A_3, B_3] A_1)^n = T([B_1, A_1] A_1)^n  = TB_1 A_1^n B_1^{-1},$$ 
which recovers (\ref{eq:Xtriv1}); that is, (\ref{eq:Xtriv1}) is redundant in this system. As such $\F_\pm$ is the set of tuples satisfying
\begin{subequations}
\label{eq:pm}
 \begin{align}
[A_3, B_3] & =  [B_1, A_1]\label{eq:pm0}\\
[A_2, T]  = [B_2, T] = [A_2, B_2] & =  1\label{eq:pm1}\\
T & =  \pm B_1 A_1^{-n} B_1^{-1}\label{eq:pm2}.
\end{align}
\end{subequations}
together with the open condition that $T \notin \left\{1, -1 \right\}$. The closure $\overline{\F}_{\pm}$ is therefore defined by the same set of equalities (\ref{eq:pm}), but without this latter open condition on $T$. As one can readily check, this contains the representations $(1, (1, 1)_i)$ and $(-1, (1, 1)_i)$. 

Fix any tuple $(T, (A_i, B_i)_i) \in \overline{\F}_\pm$ in this closure; we will show we can connect it by a path in this closure to $(\epsilon, (1, 1)_i)$ for $\epsilon \in \left\{-1, 1 \right\}$. The components $A_3, B_3$ are only constrained by (\ref{eq:pm0}) so, since fibers of the commutator map are connected (see Corollary \ref{cor:com}), we can find a path in $\overline{\F}_\pm$ from our initial tuple to one with $A_1 = B_3$ and $B_1 = A_3$; such elements automatically satisfy (\ref{eq:pm0}). Likewise, in $\SU(2)$ all centralizers are connected, so since $A_2, B_2$ are only constrained by (\ref{eq:pm1}), we can find a path in this closure from our tuple to one with $A_2 = B_2 = 1$. In this way, we are reduced to determining the connectedness of the space of $(T, A_1, B_1)$ satisfying (\ref{eq:pm2}). For this, take any path $A_1(t)$ from $A_1$ to any element of $\SU(2)$ whose $n$th power is $\pm \epsilon$. Leaving $B_1$ fixed for now, define $T(t) \defeq \pm B_1 A_1(t)^{-n} B_1^{-1}$. Then $(T(t), A_1(t), B_1, 1, 1, B_1, A_1(t))$ is a path in $\overline{\F}_\pm$ from our tuple $(T, A_1, B_1, 1, 1, B_1, A_1)$ to $(\epsilon, 1, B_1, 1, 1, B_1, 1)$. Now send $B_1$ to the identity, leaving all other components constant. To address the final claim in the lemma about the intersection with $\left\{ \epsilon \right\} \times \mathrm{Fix}_\R(\phi^*)$, use this same argument, but take the path $A_1(t) = A_1$ to be constant.
\end{proof}

We have an analogous statement for $\F_{0}$.

\begin{lemma}\label{lem:comp}
The closure $\overline{\F}_{0}$ is connected and contains $(\pm 1, (1, 1)_i)$. Any element of the intersection $\overline{\F}_{0} \cap \left\{ \pm 1 \right\} \times \mathrm{Fix}_\R(\phi^{*})$ can be connected by a path in $\left\{ \pm 1 \right\} \times \mathrm{Fix}_\R(\phi^{*})$ to the element $(\pm 1, (1, 1)_{i =1}^3)$. 
\end{lemma}

We prove this in a moment, but first we show how it finishes the proof of Theorem \ref{thm:A}: By Theorem \ref{thm:C}, for $\pm \in \left\{-,+ \right\}$ the space $\left\{ \pm 1 \right\} \times \mathrm{Fix}_\R(\phi^{ *})$ has $\lfloor n^2/2 \rfloor + 1$ connected components. By Lemmas \ref{lem:pm} and \ref{lem:comp}, exactly one of these intersects the closure of $\F_+ \cup \F_0 \cup \F_{-}$. Thus, the number of connected components of $\R(\Sigma_{\phi})$ is
$$\vert \pi_0( \R(\Sigma_{\phi}) ) \vert = \vert \pi_0(\overline{\F_+ \cup \F_0 \cup \F_{-}}) \vert + 2\big( \vert \pi_0(\mathrm{Fix}_\R(\phi^{ *})\vert - 1) = 2 \lfloor n^2/ 2 \rfloor + 1$$
as claimed.

\begin{proof}[Proof of Lemma \ref{lem:comp}]
If $\rho =(T, (A_i, B_i)_i) \in \F_{0}$, then (\ref{eq:fixed4}) implies that $A_3, B_3$ lie in the centralizer of $TX^n$, which is non-central by definition of $\F_{0}$. It follows that $A_3, B_3$ commute, $[A_3, B_3] = 1$, and so $X = A_1^n$. The identity (\ref{eq:fixed1}) implies that $[A_1, T] = 1$, and (\ref{eq:fixed2}) implies that $[B_1, T] = 1$. Since $T$ is non-central, it follows from these and (\ref{eq:fixed3}) that $A_1, B_1, A_2, B_2$ all commute with $T$. In fact, $X = A_1^n$ lies in the maximal torus containing $T$ (and $A_1, B_1, A_2, B_2$), so we see that $A_3$ and $B_3$ commute with everything in sight as well. In this case the identity (\ref{eq:fixed0}) holds trivially. To summarize, $\F_{0}$ consists of the tuples $(T, (A_i ,B_i))$ where all elements commute, together with the open conditions that $TA_1^n \neq \pm$ and $T \neq \pm 1$. The closure $\overline{\F}_{0}$ is the same set, but without these latter open conditions. Thus, this closure is the set of 7-tuples in $\SU(2)$ that all mutually commute. By Proposition \ref{prop:connectedcommutators}, this is connected. 

The intersection $\overline{\F}_{0} \cap \left\{ \pm 1 \right\} \times \mathrm{Fix}_\R(\phi^{*})$ consists of the set of tuples of the form $(\pm 1, (A_i, B_i)_i)$ where all of the $A_i, B_i$ commute. For each choice of $\pm$ this space is connected by Proposition \ref{prop:connectedcommutators}. It also contains the element $(\pm 1, (1, 1)_i)$, so the lemma follows. 
\end{proof}

\subsection{Proof of \cref{thm:B}}\label{sec:ProofOfTheoremB}

By Proposition \ref{prop:bij}, it suffices to determine the number of connected components of the extended $\R$-fixed point set $\widetilde{\mathrm{Fix}}_\R(\phi)$. Towards this end, first suppose $\rho = (A_i, B_i)_i \in \R(\Sigma)$ is abelian. As noted in Remark \ref{rem:AbelianRep}, $\rho$ is automatically in $\mathrm{Fix}_\R(\phi)$. In particular, if $T$ is any element in the centralizer of $\rho$, then $(T, \rho) \in \widetilde{\mathrm{Fix}}_\R(\phi)$. Likewise, if $(T, \rho) \in \widetilde{\mathrm{Fix}}_\R(\phi)$ for some $T \in \SU(2)$, with $\rho$ still abelian, then $T$ is in the centralizer of $\rho$. Since $\rho$ is abelian, this centralizer is necessarily connected (it is either a circle or all of $\SU(2)$). It follows from the final observation of Remark \ref{rem:AbelianRep} that any element $(T, \rho) \in \R(\Sigma_\phi)$ with $\rho$ abelian is necessarily in the component containing the trivial representation. By Lemma \ref{lem:charrest2} we see that the surjection $\R(\Sigma_\phi) \rightarrow \widetilde{\mathrm{Fix}}_\R(\phi^*)$ induces a surjective map
$$\pi_0(\R(\Sigma_\phi)) \longrightarrow \pi_0(\widetilde{\mathrm{Fix}}_\R(\phi^*)).$$
The preimage of the point associated to the trivial representation consists of a single point, while the preimage of every other points consists of two points. Then \cref{thm:B} follows from \cref{thm:A}.


\bibliographystyle{alpha}

\end{document}